\documentclass[letterpaper, 10 pt, conference]{ieeeconf}  

\IEEEoverridecommandlockouts
 \usepackage{amsmath, amssymb, bbm, xspace}
 \usepackage{amsfonts,mathrsfs}
 \usepackage{mathtools}
 \usepackage{color}
 \usepackage{graphicx}
 \usepackage{epsfig}
 \usepackage{algorithm}
 \usepackage{algorithmicx}
 \usepackage[acronym]{glossaries}
 \usepackage{subfigure}
 \usepackage{cancel}
 \usepackage{empheq}
 \usepackage{placeins}
 \let\labelindent\relax
 \usepackage{enumitem}
 \usepackage{textgreek}
 \usepackage{mathbbol}
 \usepackage{stackengine}
\usepackage{placeins}
\usepackage{dblfloatfix}
\usepackage{cancel}
\usepackage{comment}
\usepackage[hidelinks]{hyperref}

 \def\QEDhereeqn{\eqno\let\eqno\relax\let\leqno\relax\let\veqno\relax\hbox{\QED}}
 \def\QEDopenhereeqn{\eqno\let\eqno\relax\let\leqno\relax\let\veqno\relax\hbox{\QEDopen}}
 \newcommand{\bs}{\boldsymbol}
 \newcommand{\mc}{\mathcal}
 \renewcommand{\emph}{\textit}
 \newcommand\numberthis{\addtocounter{equation}{1}\tag{\theequation}}
 \newcommand{\0}{\bs 0}
 \def\1{{\bs 1}}
 \def\argmin{\mathop{\rm argmin}}
\def\min{\mathop{\rm min}}

 \newcommand{\proj}{\mathrm{proj}}
 
 \def\R{\mathbb{R}}

 \DeclareSymbolFontAlphabet{\mathbbm}{bbold}
 \DeclareSymbolFontAlphabet{\mathbb}{AMSb}%
 
 \stackMath
	 \newcommand\tsup[2][2]{%
	 	\def\useanchorwidth{T}%
	 	\ifnum#1>1%
	 	\stackon[-.5pt]{\tsup[\numexpr#1-1\relax]{#2}}{\scriptscriptstyle\sim}%
	 	\else%
	 	\stackon[.5pt]{#2}{\scriptscriptstyle\sim}%
	 	\fi%
	 }

\makeglossaries
\newacronym{KKT}{KKT}{Karush--Kuhn--Tucker}
\newacronym{ADMM}{ADMM}{alternating direction method of multipliers}
\newacronym{OPF}{OPF}{optimal power flow}
\newacronym{OPFP}{OPFP}{optimal power flow problem}
\newacronym{NUM}{NUM}{network utility maximization}
\newacronym{LMI}{LMI}{linear matrix inequality}
\newacronym{BMI}{BMI}{bilinear matrix inequality}
\newacronym{LM}{LM}{Lyapunov-Metzler}
\newacronym{SDP}{SDP}{semidefinite programming}
\newacronym{LTI}{LTI}{linear time invariant}
\newacronym{MJLS}{MJLS}{Markov jump linear system}
\newacronym{PID}{PID}{proportional-integral-derivative}
\newacronym{PPA}{PPA}{proximal-point algorithm}
\newacronym{PPPA}{PPPA}{preconditioned proximal-point algorithm}
\newacronym{PPP}{PPP}{preconditioned proximal-point}
\newacronym{NE}{NE}{Nash equilibrium}
\newacronym{GNE}{GNE}{generalized Nash equilibrium}
\newacronym{v-GNE}{v-GNE}{variational GNE}
\newacronym{ISS}{ISS}{input-to-state stability}
\newacronym{OFO}{OFO}{Online Feedback Optimization}

 \def\K{{\mc K}} 
 \def\seq{\succcurlyeq}
 \newlist{thmlist}{enumerate}{1}
 \setlist[thmlist]{label=(\roman{thmlisti}), ref=\thethm(\roman{thmlisti}),noitemsep}
 \newlist{lemlist}{enumerate}{1}
 \setlist[lemlist]{label=(\roman{lemlisti}), ref=\thelem(\roman{lemlisti}),noitemsep}
 \newlist{asmlist}{enumerate}{1}
 \setlist[asmlist]{label=(\roman{asmlisti}), ref=\theassumption(\roman{asmlisti}),noitemsep,nosep,leftmargin=*} 

 \newtheorem{lemma}{Lemma}
 \newtheorem{theorem}{Theorem}
 \newtheorem{remark}{Remark}
 \newtheorem{assumption}{Assumption}
  
\newtheorem{corollary}{Corollary}
\newtheorem{proposition}{Proposition}

\usepackage{etoolbox}
\patchcmd{\smallmatrix}{\thickspace}{\kern.5em}{}{}

\usepackage{caption}
\usepackage{booktabs}
\usepackage{multirow}

\DeclareSymbolFont{myletters}{OML}{ztmcm}{m}{it}
\DeclareMathSymbol{\uplambda}{\mathord}{myletters}{"15}

\title{\LARGE \bf
Online Feedback Optimization for Monotone Systems\\
without Timescale Separation
}

\author{Mattia Bianchi and  Florian D\"orfler
	\thanks{ M. Bianchi and Florian D\"orfler are with the Automatic Control Laboratory, ETH Zürich, Switzerland (\texttt{mbianch@ethz.ch}, \texttt{doerfler@ethz.ch}). This work is supported by ETH Zürich funds.}
}

\begin{document}

\maketitle
\thispagestyle{empty}
\pagestyle{empty}

\begin{abstract}
\gls{OFO} steers a dynamical plant to a cost-efficient steady-state,  only relying on input-output sensitivity information, rather than on a full plant model. Unlike traditional feedforward approaches, \gls{OFO}
leverages real-time measurements from the plant, thereby inheriting  the robustness and adaptability of feedback control.
Unfortunately, existing theoretical guarantees for \gls{OFO} 
assume that the controller operates on a slower timescale than the plant, which can affect responsiveness and transient performance. In this paper, we focus on relaxing this ``timescale separation'' assumption. 
Specifically, we consider  the class of monotone systems, and we prove that \gls{OFO} can achieve  an optimal operating point, regardless of the time constants of controller and plant. By  leveraging  a small-gain theorem for monotone systems, we derive several sufficient conditions for global convergence.  Notably, these conditions  depend only on the  steady-state behavior of the plant, and they are entirely independent of the transient dynamics. 
\end{abstract}

\section{Introduction}
\glsreset{OFO}
\gls{OFO} \cite{Hauswirth2021} iteratively steers a dynamical plant to an efficient operating point, by directly interconnecting the physical plant with optimization algorithms. 
Unlike feedforward approaches,  where  reference trajectories are recomputed offline at slow intervals, \gls{OFO} updates control inputs in real-time, based on measurements from the plant. This feedback mechanism enables robustness against model inaccuracies and disturbances, as well as adaptability to nonstationary environments and cost functions. For this reason, \gls{OFO} has been adopted in many engineering applications, most notably  power networks \cite{simpson2020stability,chen2020distributed,Emiliano2016} (including industrial deployment \cite{Ortmann_Deployment_2023}), but also smart buildings \cite{Belgioiosogiuseppe2021}, communication systems \cite{wang2011control}, traffic control \cite{bianchin2021time}.

Despite its practical appeal, \gls{OFO} still faces several theoretical and implementation challenges. A critical limitation, and the focus of this paper, is the reliance on the so-called  \emph{timescale separation}: it is typically assumed  that the plant evolves on a much faster timescale than the optimization algorithm/controller. Under this assumption, one can 
neglect transient dynamics, and  approximate the plant   through its steady-state input-output map \cite{Hauswirth2021}.  While this simplification facilitates the analysis, it is only valid for systems with extremely fast dynamics, such  as power grids. 

On the contrary, when controller and plant evolve on the same timescale, their dynamic interaction has to be taken into account, since it
could lead to instability, as it was shown in \cite{Hauswirthadrian2021,Belgioioso_FES_2024}. 
To deal with this issue, several works \cite{menta2018stability,Hauswirthadrian2021,Belgioioso_FES_2024,Colombino2020} enforce stability by slowing down the controller ---either by employing small gains and step sizes, or by reducing the update rate. Under these conditions, stability of the closed-loop can be studied via singular perturbation analysis.

Yet, slow controllers are not desirable. Matching the plant timescale could improve transient response, enhance adaptation to time-varying conditions, and  reduce settling time \cite{Picallo2023}. This raises a fundamental question: is it possible to guarantee stability for \gls{OFO} without requiring the controller to be slower than the plant?

In our recent work \cite{Bianchi_InfOfo}, we  proposed a condition to ensure stability for \gls{OFO} without requiring any timescale separation, based on a $\max$-type Lyapunov function. In this paper, we take a  further step  forward, by adopting a different approach, based on the theory of monotone (or ``order preserving'') systems. 
Monotone systems are a prominent class of dynamical models, with widespread applications, particularly in systems biology \cite{Angeli2014,Angeli2004Detection,Nikolaev2016}, but also in compartmental dynamics, such as in epidemiology or chemical reaction networks \cite{haddad2010nonnegative}. Notably, they also encompass the  class of positive linear systems \cite{Angeli_MCS_2003}.

\emph{Contributions}: In this paper, we consider an optimal steady-state control problem for a monotone plant with box input constraints.  
Under suitable assumptions,  we show that \gls{OFO} drives the plant to its optimal equilibrium \emph{for any value of the control gain}, 
regardless of the relative timescales of the controller and plant.

Our main convergence criterion (Theorem~\ref{th:main}) relies on a small-gain result for monotone systems \cite{Angeli_MCS_2003}. In addition, we derive several more explicit and easier to verify sufficient conditions.
 As in our earlier work \cite{Bianchi_InfOfo}, these conditions can often be enforced by introducing a sufficiently large regularization in the control objective (Remark~\ref{rem:regularization}). However, differently from \cite{Bianchi_InfOfo}, the bounds we obtain are not only independent of the time constants of plant and controller, but  are also unaffected by any other transient property, as they depend solely on  the steady-state maps.

The paper is organized as follows. 
In Section~\ref{sec:notation}, we review some necessary background on monotone systems. In Section~\ref{sec:problemstatement}, we discuss the main idea of \gls{OFO} and present our setup. Section~\ref{sec:main} introduces our main convergence criterion. In Section~\ref{sec:special}, we provide several sufficient conditions, which are easier to verify and interpret. 
Section~\ref{sec:numerics}  illustrates our results via numerical examples. 
Finally, in Section~\ref{sec:conclusion}, we discuss possible extensions and research directions. All the proofs are provided in the appendix.

\subsection{Notation and background}\label{sec:notation}
For a differentiable map $f:\R^n \rightarrow \R^m$, 
$\nabla f(x) \in \R^{m\times n}$ denotes its Jacobian, i.e.,  the matrix of partial derivatives of $f$ computed at $x\in \R^n$. If $f$ is scalar (i.e., $m=1$), we also denote by $\nabla f(x) \in\R^n$ its gradient at $x$, with some abuse of notation but no ambiguity. If $x'$ is a subset of variables, then $\nabla_{x'} f$ denotes its Jacobian/gradient with respect to $x'$. We denote by $x_i$ or $[x]_i$ the $i$-th scalar component of a vector $x$. 
For a  closed convex set $S\subseteq \R^n$,  $\proj_{S}:\R^n \rightarrow S$ is the Euclidean projection onto $S$;
 $\mathrm{T}_S:S \rightrightarrows \R^n:x\mapsto \operatorname{cl}({\bigcup_{\delta>0}\frac{1}{\delta}(S-x)})$ is the tangent cone  of $S$, where $\operatorname{cl}(\cdot)$ denotes the set closure.
The projection on the tangent cone of $S$ at $x$ is $\Pi_{S}(x,v):=\proj_{\mathrm{T}_S(x)}(v)=\lim_{\delta\rightarrow 0^+}\textstyle \frac{\proj_{S}(x+\delta v)-x}{\delta}$.


\emph{Projected dynamical systems }\cite{Cojocaru2003}:
Given a map $\mc{F}:\R^n\rightarrow \R^n$ and a closed convex set $S\subseteq \R^n$, we consider the projected dynamical system
\begin{align}\label{eq:projdyn} \dot{x}=\Pi_S(x,\mc{F}(x)), \qquad x(0)=x_0\in S.
\end{align}
In \eqref{eq:projdyn}, the projection operator is possibly discontinuous on the boundary of
$S$. If $\mc{F}$ is Lipschitz on $S$, then the system \eqref{eq:projdyn} admits a unique global Carathéodory solution, i.e., there exists a unique absolutely continuous function $x:\R_{\geq 0}\rightarrow \R^n$ such that $x(0)=x_0$, $\dot{x}(t)=\Pi_S(x,\mc{F}(x))$ for almost all $t$. Moreover, $x(t)\in S$ for all $t\geq 0$, as on the boundary of $S$ the projection operator
restricts the flow of $\mc{F}$ such that the solution of \eqref{eq:projdyn} remains in $S$ (while $\Pi_S(x,\mc{F}(x))=\mc{F}(x)$ if $x\in \rm{int}(S)$).

\emph{Monotone control systems} \cite{EncisoSontag_Global2006}: Let $\mathcal{K}\subset \R^n$ be an orthant, i.e., $\K = \{x \in \R^n \mid  (-1)^{\epsilon_i} x_i  \geq 0 , \forall i \in \{1,2,\dots,n\} \}$, for some binary vector $(\epsilon_1,\epsilon_2,\dots\epsilon_n)\in \{0,1\}^n$. Then, $\K$ induces
the partial order ``$\succcurlyeq_{\mc K}$''  in $\R^n$, where 
 \begin{align}
    x \seq_{\mc K} y \iff x - y \in \mc K. \end{align}
For instance, the nonnegative orthant $\K = \R^n_{\geq 0}$ induces the standard (elementwise) ordering. 
Let us consider a dynamical system 
\begin{align}\label{eq:introsystem}
    \dot x = f (x,u), \quad y = g(x),
\end{align} with state space $X \subseteq \R^n$, input space $U \subseteq \R^m$, output space $Y \subseteq \R^p$. For a measurable, locally bounded input function $u:\R_{\geq 0}\rightarrow U$ and an initial condition $x_0 \in \R^n$, let $x(t,x_0,u) \in X$  be the corresponding solution of   \eqref{eq:introsystem} at time $t$ (assuming it is well defined; we always consider solutions in the Carath\'eodory sense in this paper). Consider the orders induced by  three orthants  $\K_U \subset \R^m$, $\K_X \subset \R^n$, $\K_Y \subset \R^p$. Then, the dynamical system is \emph{monotone} if: (1) for any two inputs $u,u'$ such that $u(t) \seq_{\K_U} u'(t)$ for almost all $t$, and for any two initial conditions $x_0 \seq_{\K_X} x_0'$, it holds that: \newline (1) for all $t \geq 0$, 
\begin{align}
x(t,x_0,u) \seq_{K_X} x(t,x_0',u');
\end{align}
(2) the output map $g$ is nondecreasing, namely \begin{align}
    x \seq_{K_X} x'  \Rightarrow g(x) \seq_{\K_Y} g(x').
\end{align} 
There are several ways to check if a system is monotone without computing its solutions, e.g., see \cite[§2.1]{Angeli2004}. For instance, a smooth system is monotone with respect to the standard ordering ($\K_U =\R^m_{\geq 0}$ $, \K_X= \R^n_{\geq 0}$, $\K_Y =\R^p_{\geq 0}$) if and only if, for all $x\in X$, $u\in U$,
\begin{align}\label{eq:MonCondition}
\begin{aligned}\textstyle (\forall i\neq j) &  \textstyle  \qquad \frac{\partial f_i}{\partial x_j} (x,u) \geq 0  
\\
(\forall i,j,l) & \textstyle  \qquad  \frac{\partial f_i}{\partial u_j} (x,u) \geq 0, \quad \frac{\partial g_l}{\partial x_i}(x) \geq 0.
\end{aligned}
\end{align}
In particular, a linear system $\dot x = Ax+Bu$, $y = Cx$, is monotone if $A$ is Metzler and $B,C$ are nonnegative. 

We say that system \eqref{eq:introsystem} has a  steady-state map $k_x :U\rightarrow X$ if, for every constant input $\hat u(t) \equiv \bar u  \in U$ and independently of $x_0\in\R^n$, the solution $x(t,x_0,\hat u )$ converges to $k_x ( \bar u)$ as $t\rightarrow \infty$. In this case, we define the  steady state output map $k_y$ as $k_y = g\circ k_x$.

We will use the following result, which can be obtained as a special case of \cite[Th.~2]{EncisoSontag_Global2006} (see  Appendix~\ref{app:proofofSGT} for a proof). 
\medskip

\begin{lemma}[Small-gain theorem]\label{lem:SGT}
    Let the following conditions hold:
    \begin{enumerate}
        \item $U = Y = \R^m$, and $X$ is a convex closed (possibly unbounded) box, i.e., $X$ is the  Cartesian product of closed intervals in $\R$;
        \item the system \eqref{eq:introsystem} is monotone with respect to some orthants $\K_U,\K_X,\K_Y$, and $\K_Y = -\K_U$ (i.e., the output is ordered by the opposite order of the input); 
        \item The system admits a continuous steady-state map $k_x$; furthermore, $g$ is continuous. 
        \item (\emph{small-gain condition}) For any $u_0\in U$, the iteration
        \begin{align}
            u_{n+1} = k_y(u_n) = g(k_x(u_n)) 
        \end{align}
        converges to a unique $u^\star \in U$. 
    \end{enumerate}
    Then, all bounded solutions of the closed-loop system
    \begin{align*}
        \dot x = f(x,g(x))
    \end{align*}
    converge to $x^\star = k_x (u^\star) $.
\end{lemma}


\medskip

\section{Problem statement} \label{sec:problemstatement}
We consider the dynamical system (or plant)
\begin{subequations}\label{eq:plant}
    \begin{align}
        \dot x & = f(x,u) \\
        y & = g(x),
    \end{align}
\end{subequations}
with state $x\in\R^n$, input $u\in \R^m$, and output $y\in \R^p$. We assume that the plant has a well defined steady-state map: this is a fundamental assumption in  \gls{OFO}  \cite[Asm.~2.1]{Hauswirthadrian2021}. 

\begin{assumption}\label{asm:plant}
   The mappings $f:\R^n \times \R^m \rightarrow \R^n$ and $g:\R^n \rightarrow \R^p$ are locally Lipschitz. 
   The plant \eqref{eq:plant} admits a well defined  steady-state map $k_x:\R^m\rightarrow \R^n$ and steady-state output map $k_y \coloneqq g\circ k_x:\R^m \rightarrow \R^p$: for any constant input $\hat u(t) \equiv \bar u\in \R^m$, the solution to  \eqref{eq:plant} converges to $k_x(\bar u)$.  
   Furthermore, $k_x$ is continuous, $k_y$ is continuously differentiable, and the \emph{sensitivity} $\nabla k_y : \R^m \rightarrow \R^{p\times m}$ is locally Lipschitz. 
\end{assumption}

The control objective is to steer the plant \eqref{eq:plant} to a solution of the optimization problem 
\begin{subequations}
\label{eq:opt1}
\begin{align}
    \min_{u\in \mathcal U,y\in\R^p} & \quad   \Phi(u,y) \coloneqq \Phi_u(u) + \Phi_y(y)
    \\[-0.3em]
    \textnormal{s.t.}   & \quad   y = k_y(u),
\end{align}
\end{subequations}
where $\Phi$ is a cost function, $\mathcal{U}$ is the set of feasible inputs, and the steady-state constraint  $y = k_y(u)$ ensures that any solution to \eqref{eq:opt1} is an input-output equilibrium for the plant \eqref{eq:plant}. Equivalently, \eqref{eq:opt1} can be recast as
\begin{align}\label{eq:opt2}
    \min_{u\in \mathcal U} & \quad  \tilde \Phi(u),
\end{align} 
where $\tilde \Phi(u)  \coloneqq \Phi(u,k_y(u))$. For technical reasons, we restrict our attention to box-constrained problems. 
 \begin{assumption}\label{asm:cost}
     The function $\Phi:\R^m\times \R^p \rightarrow \R$ is continuously differentiable, and its gradient  $\nabla \Phi$ is locally Lipschitz. The set $\mathcal U \subset \R^m$ is a compact convex box set, i.e., $\mathcal{U} = [u_1^{\min},u_1^{\max}] \times [u_2^{\min},u_2^{\max}] \times \dots \times [u_m^{\min},u_m^{\max}]$. 
 \end{assumption}

We note that Assumption~\ref{asm:cost} ensures existence of a solution for \eqref{eq:opt2}, hence of \eqref{eq:opt1}. 
The projected gradient flow for \eqref{eq:opt2} is $\dot u = \Pi_{\mathcal U}(u,-\alpha \nabla \tilde \Phi (u))$, or, by the chain rule,
\begin{align}\label{eq:gradientflow}
    \dot u = \Pi_{\mathcal U} \left(u, - \alpha H(u) \nabla \Phi (u,k_y(u)) \right)
\end{align}
where $\nabla \Phi(u,y) = \begin{bmatrix} \nabla \Phi_u(u)^\top & \nabla \Phi_y(y)^\top \end{bmatrix}^\top$,
\begin{align}
    H(u) \coloneqq
   \begin{bmatrix}
       I_m &  \nabla k_y(u)^\top
    \end{bmatrix},
\end{align}
 $\alpha >0$ is a tuning gain, and the projection on the tangent cone can be computed explicitly as
\begin{align}
    [\Pi_{\mathcal U}(u,v)]_i = \begin{cases}
        v_i & \text{if } u_i^{\min} < u_i < u_i^{\max}
        \\
        \max\{0,v_i\} & \text{if } u_i = u_i^{\min}
        \\
        \min\{0,v_i\} & \text{if } u_i = u_i^{\max}.
    \end{cases}\label{eq:explicit_ptc}
\end{align}

\gls{OFO} proposes to replace the open-loop gradient flow in \eqref{eq:gradientflow} with a closed-loop controller,
obtained by substituting  the unknown steady-state output map $k_y$ with measurements from the plant \eqref{eq:plant}, namely 
\begin{align}\label{eq:OFO}
    \dot u = \Pi_{\mathcal U} \left (u,- \alpha  H (u) \nabla \Phi (u,y) \right) .
\end{align}
This feedback approach guarantees continuous adaptation (e.g., to time-varying disturbances), while only requiring the knowledge of the sensitivity $\nabla k_y$. Importantly, the equilibria of the closed-loop system \eqref{eq:plant},\eqref{eq:OFO} correspond to the critical points of \eqref{eq:opt2} (i.e., points $u^\star$ such that $\langle \nabla \tilde \Phi(u^\star), u -u^\star \rangle \geq 0$ for all $u \in \mathcal U$). 

Under opportune assumptions, closed-loop stability can also be ensured, by choosing a small-enough gain $\alpha$ for the controller \eqref{eq:OFO}---based on singular-perturbation analysis \cite{Hauswirthadrian2021,Colombino2020}. The idea is that, for small $\alpha$, the time constant of the controller becomes much larger than that of the plant, and thus the plant is always approximately at steady-state. In this case, $y \approx k_y(u)$, and the OFO controller \eqref{eq:OFO} approximates the gradient flow \eqref{eq:gradientflow}.

On the downside, enforcing this timescale separation affects the responsiveness of the controller \cite{Picallo2023}, which could take a long time to steer the system to an equilibrium or adapt to changes in the operation (e.g.,  time-varying costs or disturbances), resulting in degraded transient performance. 
Hence, it would be highly valuable to allow controller and plant to evolve on the same timescale (i.e., enabling  large values of $\alpha$), without jeopardizing asymptotic optimality or stability. 

In this paper, we aim to relax timescale separation by exploiting some structure of the plant \eqref{eq:plant}. 
In particular, we focus on monotone systems, as postulated next. 
\begin{assumption}\label{asm:monotone}
    The plant \eqref{eq:plant} is monotone with respect to the order induced by three orthants  $\K_U\subset \R^m$, $\K_X\subset \R^n$, $\K_Y\subset \R^p$.
\end{assumption}

\section{OFO for monotone systems}
\label{sec:main}
In this section, we show that monotonicity of the plant in Assumption~\ref{asm:monotone} can be exploited to completely avoid  timescale separation  in \gls{OFO}. Specifically, under suitable conditions, the closed-loop \eqref{eq:plant},\eqref{eq:OFO} converges to an equilibrium \emph{for any value of $\alpha$}. 

We start by showing that the interconnected system is well posed. All the proofs are provided in the appendix.

\begin{proposition}\label{prop:wellposedness}
    For any $x(0)\in \R^n$, $u(0) \in \mathcal U$, the closed-loop system \eqref{eq:plant},\eqref{eq:OFO} admits a unique global solution $x(t), u(t)$. Furthermore, the solution is bounded, and $u(t) \in \mathcal{U}$ for all $t\geq 0$. 
\end{proposition}

Let us now rewrite the closed-loop system
\eqref{eq:plant},\eqref{eq:OFO} 
as
\begin{subequations}\label{eq:clreformulated}
\begin{align}
  \dot x & = f(x,u)   \label{eq:clreformulatedx}
  \\
  \dot v & =  \Pi_{\mathcal U} \left (v,- \alpha  H (v) \nabla \Phi (v,g(x)) \right), \quad z = v \label{eq:clreformulatedv}
  \\
  & u = z \quad \text{(feedback loop),}\label{eq:clreformulatedf}
\end{align}
\end{subequations}
with input $u$, output $z$, and unitary feedback $u =z$. This form allows us to exploit the small-gain theorem in Lemma~\ref{lem:SGT}, under the following assumption. 
\begin{assumption}\label{asm:core}
    The following hold:
    \begin{enumerate}
        \item[(i)] The subsystem \eqref{eq:clreformulatedv} is monotone with respect to   $\K_X$, $\K_V$, $\K_Z = - \K_U$, where $\K_U$ and $\K_X$  are from Assumption~\ref{asm:monotone}, and $\K_V\subset \R^m$ is some orthant: namely, if $x(t) \seq_{\K_X} x'(t) $ for  all $t$ and $v_0 \seq_{\K_V} v_0'$, then $v(t,v_0,x) \seq_{\K_V} v(t,v_0',x') $ for all $t$, and furthermore $v \seq_{\K_{V}} v' \Rightarrow v \seq_{\K_{Z}} v'$;
        \item[(ii)] The subsystem \eqref{eq:clreformulatedv} admits a well defined steady-state map $k_v:\R^n \rightarrow \R^m$: for any constant $x$, the solution to \eqref{eq:clreformulatedv} converges to $k_v (x)$. Furthermore, $k_v$ is continuous;
        \item[(iii)] For any $u_0 \in \mathcal U$, the iteration 
        \begin{align}\label{eq:iterationmain}
            u_{n+1} = \argmin_{u\in \mathcal U} \Phi_u(u) + k_y(u)^\top \nabla \Phi_y (k_y(u_n))
        \end{align}
        converges to a unique $u^\star \in U$. 
        \end{enumerate}
\end{assumption}

The following is our main convergence result.

\begin{theorem}\label{th:main}
    Let Assumptions~\ref{asm:plant}, \ref{asm:cost}, \ref{asm:monotone}, \ref{asm:core} hold.  Then, for any $x_0\in \R^n$, $u_0 \in\mathcal U$, the solution of \eqref{eq:clreformulated} converges to $(x^\star,u^\star)$, where $u^\star$ is the unique minimizer of \eqref{eq:opt2} and $x^\star = k_x (u^\star)$, for any $\alpha >0$.
\end{theorem}

\begin{remark}
     Assumption~\ref{asm:core}(ii) ensures that the $\argmin$ in Assumption~\ref{asm:core}(iii) is a singleton (in fact, the iteration in \eqref{eq:iterationmain} equals  $u^{n+1} = k_v(k_x (u_n))$, as shown in the proof of Theorem~\ref{th:main}). 
\end{remark}

\begin{remark}
     The conditions in Assumption~\ref{asm:core} imply the existence of a unique equilibrium for \eqref{eq:clreformulated}, hence that \eqref{eq:opt2} has  a unique critical point $u^\star$. Hence, $u^\star$ must be the global minimum (which exists by compactness of $\mc{U}$). While it might be possible to extend this result to problems with multiple critical points (e.g., based on the results in \cite{Malisoff2006}), we leave this extension for future work.
\end{remark}
\begin{remark}
    Theorem~\ref{th:main} ensures global attractivity, but not (Lyapunov) stability for the equilibrium $(x^\star,u^\star)$. Stability could be  studied under additional  conditions \cite[\S 4]{EncisoSontag_Global2006}. 
\end{remark}

The conditions in Assumption~\ref{asm:core} aim for broad applicability, but are difficult to interpret. Therefore, we  provide sufficient conditions for some concrete cases of interest in the following section.

\section{Sufficient conditions}\label{sec:special}

In this section, we assume that the plant \eqref{eq:plant} is monotone with respect to the standard ordering, i.e., that 
\begin{align}
    \K_U = \R^m_{\geq 0}, \quad  \K_X = \R^n_{\geq 0}, \quad \K_Y = \R^p_{\geq 0};
\end{align} this is actually without loss of generality \cite{Angeli_MCS_2003}. For brevity, we write ``$ \seq$'' rather than ``$\seq_{\R^\cdot _{\geq 0} }$'' when referring to the standard ordering of opportune dimension. In the following, we provide some sufficient criteria to verify the conditions in Assumption~\ref{asm:core}.
\medskip

\begin{lemma}\label{lem:asmi}
   Let Assumptions \ref{asm:plant}, \ref{asm:cost}, \ref{asm:monotone} hold. Then, Assumption~\ref{asm:core}(i) is verified with respect to  $\K_X = \R^n_{\geq 0} $, $\K_V =\R^{m}_{\leq_0}$, $\K_Z = - \K_U =  \R^{m}_{\leq_0} $, if either of the following holds:
    \begin{enumerate}
        \item[(i)] The plant \eqref{eq:plant} is single-input single-output (SISO), i.e., $m=p = 1$. The function $\Phi_y$ is convex;
        \item[(ii)] The steady-state output map $k_y$ is affine, i.e.,  $k_y(v)= Sv +s$ for some $S\in \R^{p\times m}$, $s\in\R^p$. The cost $\Phi_u$ is twice differentiable and its Hessian $\nabla^2 \Phi_u(v)$ has nonpositive off-diagonal entries, for any $v \in \R^m$.  The cost $\Phi_y$ is twice differentiable and  $\nabla^2 \Phi_y(y)$ has all nonnegative entries, for all  $y \in \R^p.$  
    \end{enumerate}
\end{lemma}

\medskip

\begin{lemma}\label{lem:asmii}
Let Assumptions~\ref{asm:plant}, \ref{asm:cost}, \ref{asm:monotone} hold.  Then, Assumption~\ref{asm:core}(ii) is verified if either of the following holds: 
\begin{enumerate}
\item[(i)] The function  $\Psi(v,x) \coloneqq\Phi_u(v) +k_y(v)^\top \nabla \Phi_y (g(x))$ is strictly convex in $v$, for any fixed $x\in \R^n$;
\item[(ii)] The function $\Phi_u$ is strictly convex. The steady-state output map $k_y$  is affine. 

\end{enumerate}
\end{lemma}

\medskip

\begin{lemma}\label{lem:asmiii}
Let Assumptions~\ref{asm:plant}, \ref{asm:cost}, \ref{asm:monotone} hold.  Then, Assumption~\ref{asm:core}(iii) is verified if the following  holds:
the function $\tilde \Psi(u,\bar u) \coloneqq\Phi_u(u) +k_y(u)^\top \nabla \Phi_y (k_y (\bar u))$ is $\mu$-strongly convex in $u$, for any  $\bar u\in \mathcal U$; the map $\nabla_u \tilde \Psi(u,\bar u)$ is $\ell$-Lipschitz in $\bar u$, for every fixed $ u\in \mathcal U$; furthermore, $\mu > \ell$.
\end{lemma}

\medskip
An example of systems that have affine (actually, linear) steady-state output map is that of  stable linear time-invariant (LTI) systems of the form $\dot x = Ax+Bu, y = Cx$, for which $k_y(u) = -CA^{-1}Bu$. For convex quadratic costs $\Phi_u(u) = u^\top Q_u u$ and  $\Phi_y(y) = y^\top Q_y y$, (with symmetric $Q_u$, $Q_y$), condition (ii) in Lemma~\ref{lem:asmi} can be immediately checked by looking at the elements of $Q_u, Q_y$.

Based on the sufficient conditions in the previous three lemmas, we can derive  simple convergence guarantees, as exemplified next, and further in Section~\ref{sec:numerics}. 


\begin{corollary}\label{cor:linear}
Let Assumptions~\ref{asm:plant}, \ref{asm:cost}, and \ref{asm:monotone} hold.
    Let   $k_y(u) = Su  $ be linear,  $\Phi_u(u) = \beta_u u^\top u $, $\Phi_y(y) = \beta_y y^\top  y$,
    for some $\beta_u,\beta_y >0$.  
    Assume that   $\beta_u > \beta_y  \|S\|^2$.
    Then the closed loop system \eqref{eq:plant},\eqref{eq:OFO} converges to a solution of \eqref{eq:opt1}, for any $\alpha >0$. 
\end{corollary}

The previous corollary should only serve as an illustration, and it can be generalized (e.g., to separable non quadratic $\Phi_u(u)$); see also Section~\ref{sec:numerics}.

\begin{remark}\label{rem:regularization}
    Under Assumptions~\ref{asm:plant}, \ref{asm:cost}, \ref{asm:monotone}, the conditions in Assumption~\ref{asm:core}(ii) and (iii) can  always be enforced.   In fact, due to monotonicity and the compactness of $\mathcal U$, if the initial condition $x_0$ is contained in a bounded set,  we can infer that all the possible trajectories (independently of $u \in \mathcal U$) of the plant \eqref{eq:plant} are contained in a bounded set $\mathcal{X}$ (see the proof of Proposition~\ref{prop:wellposedness} in Appendix~\ref{app:proofpropWellposedness}). On the compact sets  $\mathcal{U}$ and $\operatorname{cl}(\mathcal{X})$, by continuity of the involved functions: 
    
    (A)  $\nabla k_y(u)$ is Lipschitz, since it  is locally Lipschitz by Assumption~\ref{asm:plant}, and $\nabla\Phi_y (g(x))$ is bounded: hence $k_y(u)^\top \nabla\Phi_y (g(x))$ is smooth in $u$ for any fixed $x$; 
    
    (B) $\nabla k_y(u)$ is bounded, and $\nabla\Phi_y (k_y(\bar u))$ is Lipschitz in $\bar u$ (since it is locally Lipschitz): hence $\nabla k_y(u) ^\top \nabla\Phi_y (k_y(\bar u))$ is Lipschitz in $\bar u $ for any fixed $u$;

    (C)  $\nabla k_y(u)$ is Lipschitz, and $\nabla\Phi_y (k_y(\bar u))$  bounded: hence $k_y(u) ^\top \nabla\Phi_y (k_y(\bar u))$ is smooth in $u$ for any fixed $\bar u$;
    
    Hence we only need  $\Phi_u$ to be sufficiently strongly convex to satisfy the conditions in Lemma~\ref{lem:asmii}(i) (using (A)) and Lemma~\ref{lem:asmiii} (using (B) and (C)).      
    In turn, sufficient strong convexity can be  always enforced by regularization, e.g., by replacing the original cost $\Phi_u$ with $\bar \Phi_u(u) = \Phi_u(u)+ \bar \beta \|u\|^2$, with $\bar \beta>0 $ large enough.   While enabling arbitrary gain $\alpha$ can significantly improve convergence speed, the price to pay is that regularization results in suboptimal steady-state performance with respect to the original cost. This tradeoff between suboptimality and transient performance was already noted  in \cite[Rem.~1]{Bianchi_InfOfo}. One qualitative difference is that the conditions we derive here are not only independent of the time constant of the plant  \eqref{eq:plant}, but also unaffected by any information on the transient behavior---they are only related to the steady-state maps of the plant. 
\end{remark}

\section{Numerical examples}\label{sec:numerics}

We present two simple numerical examples to illustrate
our theoretical results. 

\subsection{LTI system}\label{sec:LTI}

We consider the linear system $\dot x = Ax+Bu +B_w w$, $y =Cx$, where
\begin{align*}
    A = \begin{bmatrix}
        -1 & 1 \\ 0.5 & -1
    \end{bmatrix},  B = \begin{bmatrix}
        1\\ 0
    \end{bmatrix},  B_w = \begin{bmatrix} 0.9 \\ 0 \end{bmatrix},   C = \begin{bmatrix}
        0 & 1
    \end{bmatrix},
\end{align*}
and $w \in \R$ is an unmeasured (constant) disturbance with $|w| \leq 1$. The input is constrained in the set $\mathcal U = [-0.7,1]$. The cost function is given by $\Phi(u,y) = 1.1u^2 + (y-y_{\textnormal{ref}})^2$, with $y_\textnormal{ref} = 2$. 

The steady state output map is given by $k_y(u) = -CA^{-1}[Bu +B_w w] \coloneqq Su +s$. 
The matrix $A$ is Hurwitz, and the system is monotone with respect to the standard order, by \eqref{eq:MonCondition}; thus, Assumptions~\ref{asm:plant}, \ref{asm:cost}, and \ref{asm:monotone} are satisfied.  By Lemma~\ref{lem:asmi}(i), Assumption~\ref{asm:core}(i) holds, while Lemma~\ref{lem:asmii}(ii) ensures Assumption~\ref{asm:core}(ii). Furthermore, 
$k_y(u)^\top \nabla \Phi_y(k_y(\bar u)) = 2 (Su+s)^\top (S\bar u +s - y_{\textnormal{ref}})$.
Then, the function $\tilde \Psi(u,\bar u)$ in Lemma~\ref{lem:asmiii} is $\mu$ strongly convex, with $\mu =2.2$ (because $\Phi_u$ is, and the other addend is linear in $u$). Furthermore, 
$\nabla k_y(u)^\top \nabla \Phi_y(k_y(\bar u)) = 
2  S^\top (S\bar u +s - y_{\textnormal{ref}})
$  is   $\ell$-Lipschitz in $\bar u$ for any $u\in \mc U$, with $\ell =2 \|S\|^2=2$. Hence, Lemma~\ref{lem:asmiii} ensures that Assumption~\ref{asm:core}(iii) is also verified. 
In turn, by Theorem~\ref{th:main}, the \gls{OFO} controller \eqref{eq:OFO} guarantees convergence to the solution of problem \eqref{eq:opt1} for any  $\alpha>0$ (for fixed disturbance).

We simulate the controller in \eqref{eq:OFO} for different values of $\alpha$ (note that $\nabla k_y$ is independent of the unmeasured disturbance $w$). To evaluate how quickly the controller can adapt to time-varying conditions, we choose a piecewise constant disturbance that switches between $-1$ and $1$. The resulting optimal inputs are plotted in blue in Figure~\ref{fig:LTI}. While for large values of $\alpha$ the controller quickly converges, for smaller values of $\alpha$ the controller is slow and does not track accurately the optimal value.

\begin{figure}[h]
\includegraphics[width=\columnwidth]{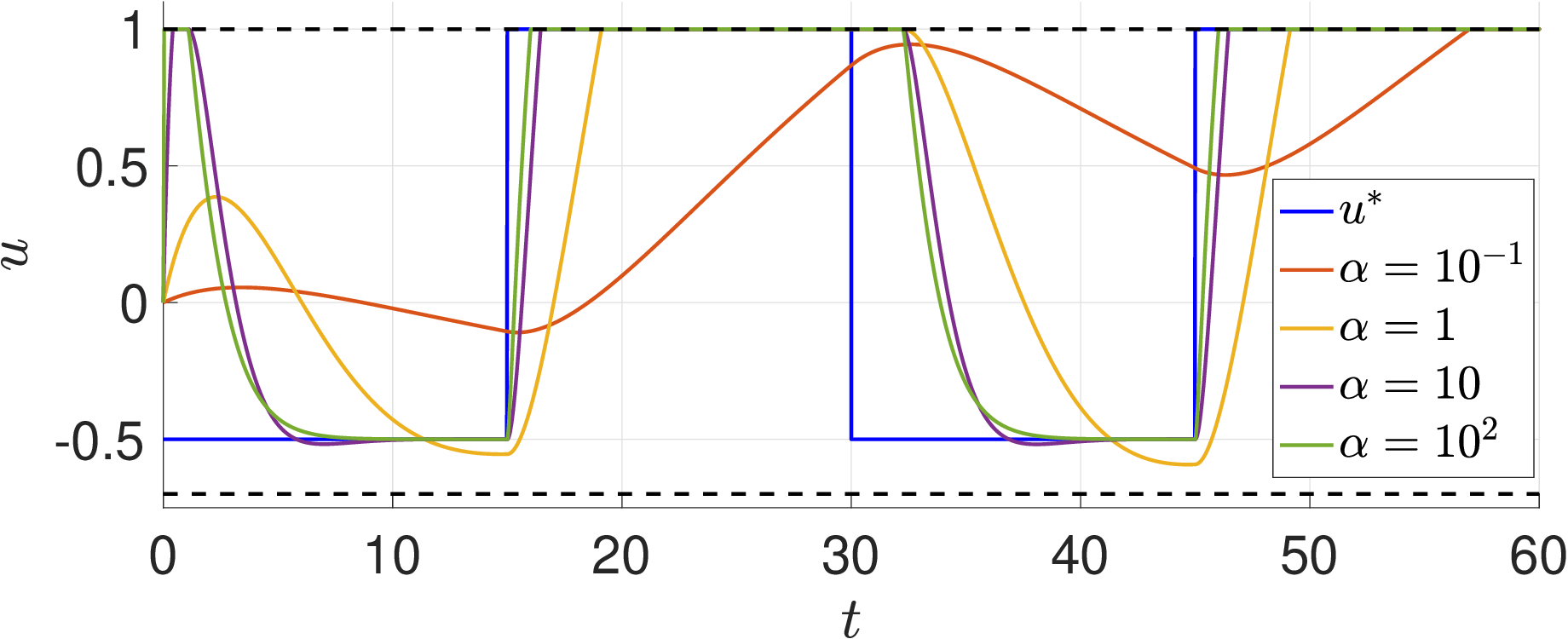}
\caption{Performance of the \gls{OFO} controller \eqref{eq:OFO} for the linear plant in Section~\ref{sec:LTI} and different values of the control gain $\alpha$. Input constraints (dotted lines) are always satisfied.} \label{fig:LTI}
\end{figure}

\subsection{Control of gene expression}\label{sec:NL}

We consider a simplified gene expression model inspired by \cite{Uhlendorf2012}. The goal is to control live yeast cells by applying osmotic shocks, so that the  
production level of a certain protein  tracks a predefined target trajectory.
The system is modeled through the following second-order positive dynamics
\begin{align}\label{eq:gene}
    \dot x_1 & = u - \gamma_1 x_1
    \\
    \dot x_2 & = \theta_2 x_1 - \gamma_2 \frac{x_2}{\theta_1+x_2}, \qquad y= x_2,
\end{align}
where $x_1$ is the recent osmotic stress, $x_2$ is the protein level, and we choose $\theta_1 =750$, $\theta_2=0.58$, $\gamma_1 = 4.02$, $\gamma_2  = 37.5$. 
The cost function is given by $\Phi(u,y) = \frac{\mu}{2} u^2 +(y-y_{\textnormal{ref}})^2$, where  $\mu =20$, $y_{\textnormal{ref}} \in \{0,1,2\}$.  The input $u$ is constrained in $\mathcal{U} = [0, 0.6] $. The system is initialized at $(0,0) $.

For the given bounds $\mathcal U$, the system is converging to an equilibrium for each fixed input, with steady-state output map $k_y(u)=\theta_1\theta_2u/(\gamma_2\gamma_1-\theta_2u)$. Furthermore,
the system is monotone with respect to the standard order, as it can be checked via \eqref{eq:MonCondition}. Hence,  Assumptions~\ref{asm:plant}, \ref{asm:cost}, and \ref{asm:monotone} are satisfied. Assumption~\ref{asm:core}(i) is satisfied by Lemma~\ref{lem:asmi}(i). In the set $\mathcal{U}$, $k_y(u)$ is convex, nonnegative,
Lipschitz with parameter $\sigma = 2.9$, and smooth  with parameter $\eta = 0.03$  (i.e., it has $\eta$-Lipschitz gradient). Consider the function $\tilde\Psi(u,\bar u) = \Phi_u(u) +2k_y(u) k_y(\bar u) -2k_y(u)y_\textnormal{ref}$ from Lemma~\ref{lem:asmiii}. Then, since $y_\textnormal{ref}\leq 2$, $k_y(u)y_\textnormal{ref}$ is $2\eta$-smooth; $k_y(u) k_y(\bar u)$ is convex; thus, $\tilde\Psi(u,\bar u) $ is $(\mu - 4\eta)$-strongly convex in $u$, for any $\bar u$. Furthermore, 
$2\nabla k_y(u) k_y(\bar u) -2\nabla k_y(u)y_\textnormal{ref}$ is $( 2\sigma^2)$-Lipschitz continuous in $\bar u$, for any $u$. Therefore, by Lemma~\ref{lem:asmiii}, Assumption~\ref{asm:core}(iii) is satisfied. Finally, 
we can show that Lemma~\ref{lem:asmii}(i) implies that Assumption~\ref{asm:core}(ii) is also satisfied. In fact, $\Phi_u$ is $\mu$-strongly convex; $k_y(v) x_2$ is convex since $k_y(v)$ is convex and $x$ is nonnegative for all times, due to the positive dynamics \eqref{eq:gene}; and $k_y(v) y_{\textnormal{ref}}$ is $2\eta$ smooth. Hence, the function $\Psi(v,x) = \Phi_u(v)+ 2 k_y(v) x_2- 2 k_y(v) y_{\textnormal{ref}}$ is $\mu -4 \eta $ strongly convex. Hence, by Theorem~\ref{th:main}, the controller \eqref{eq:OFO} guarantees convergence to the solution of \eqref{eq:opt1}, for any value of $\alpha>0$ (and fixed $y_{\textnormal{ref}}$).

Figure~\ref{fig:NL} shows the resulting trajectories for different values of $\alpha$. The reference $y_{\textnormal{ref}}$ is chosen to be piecewise constant; therefore, also the solution $(u^\star,y^\star)$ to \eqref{eq:opt1} is piecewise constant (blue lines). For smaller values of $\alpha$, the controller is slower in tracking the optimal input and outputs. For larger values of $\alpha$, while convergence is retained, there is significant overshoot. Thus, excessively large values for the control gain
might be undesirable. Interestingly, for $\alpha = 10^{-1}, 1,$ and $10^2$, the trajectories are nearly indistinguishable. In this regime, the controller operates significantly faster than the plant ---contrary to the typical setting in \gls{OFO}--- and the plant dynamics become the bottleneck for the overall convergence speed.

\begin{figure}[h]
\includegraphics[width=\columnwidth]{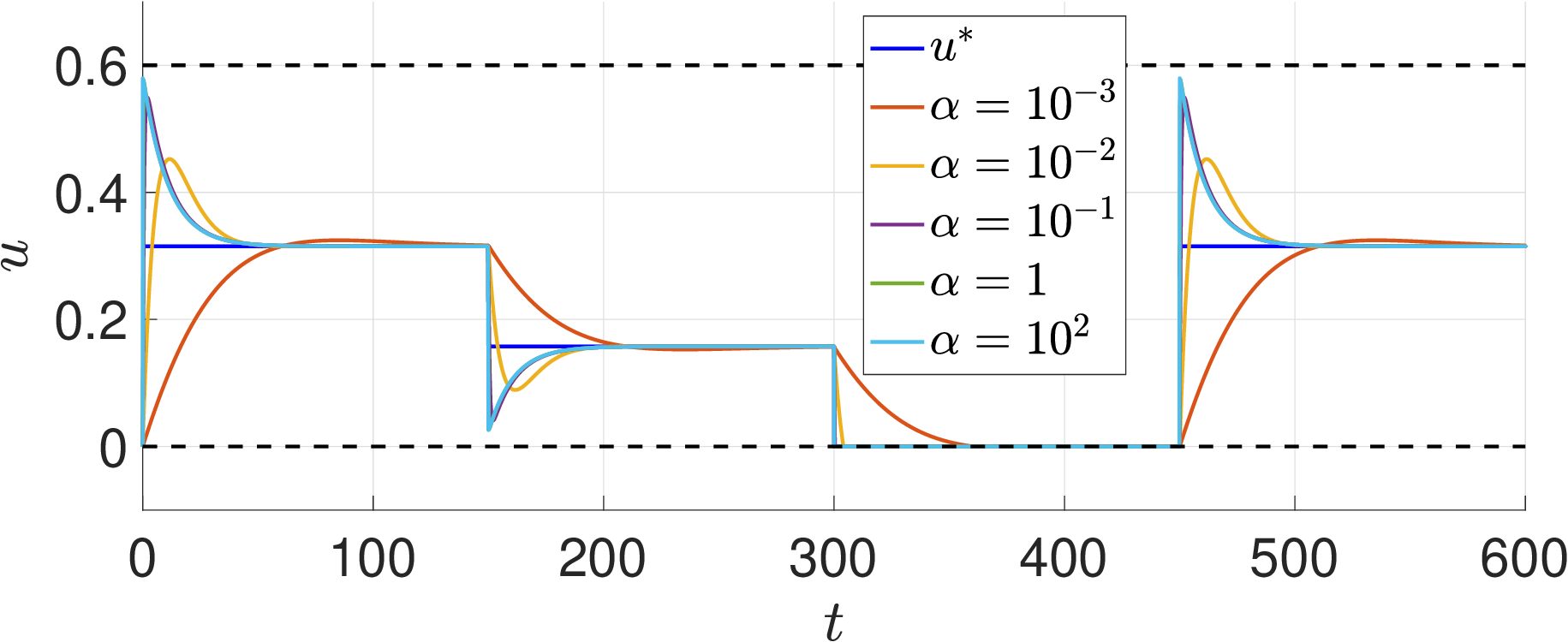}
\\[1em]
\includegraphics[width=\columnwidth]{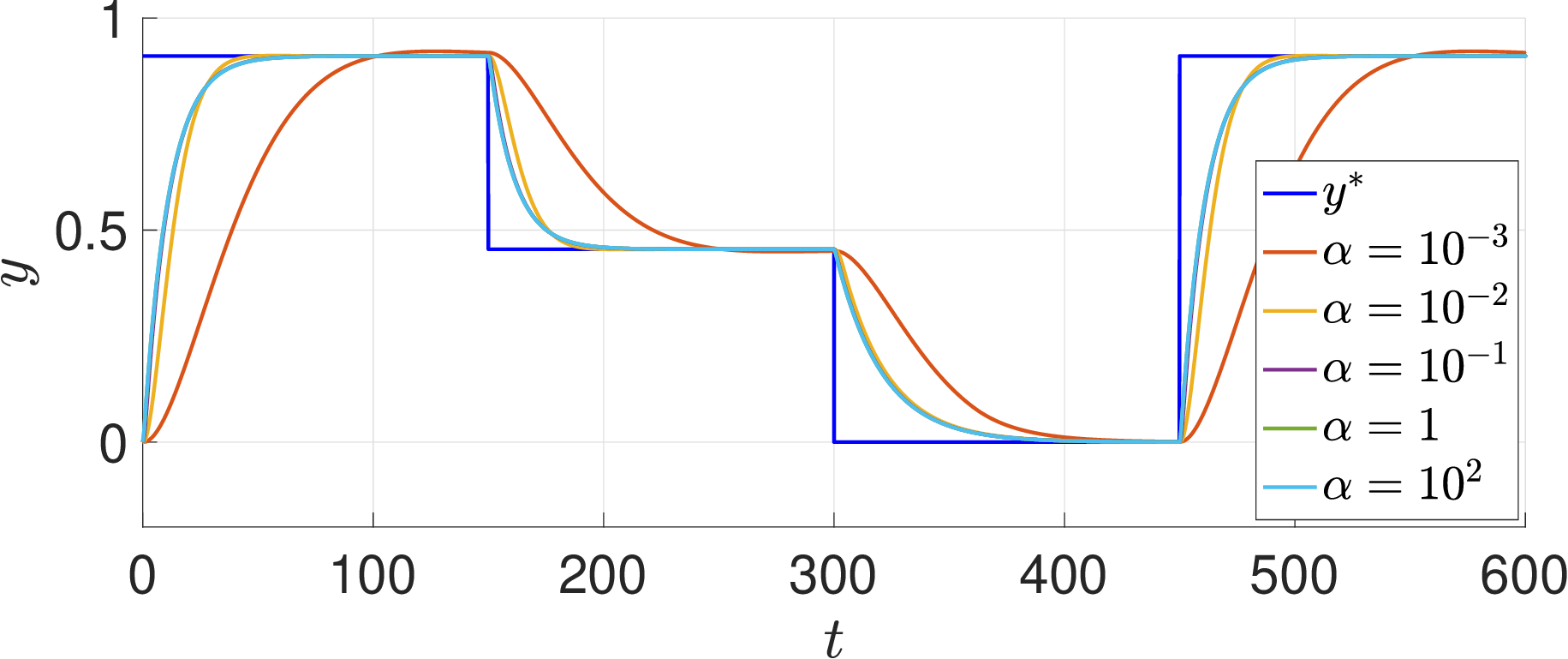}
\caption{Gene expression control for the plant in Section~\ref{sec:NL} in closed-loop with \eqref{eq:OFO}. Input constraints are shown by the dotted line.} \label{fig:NL}
\end{figure}

\section{Research directions}\label{sec:conclusion}

We showed that, under suitable assumptions, online feedback optimization can steer a monotone plant to an optimal operating point, without requiring any timescale separation. Our technical results could be extended to prove stability, rather than mere convergence. In addition, optimization problems with multiple solutions might be addressed  based on results for monotone systems with set-valued steady-state maps; this would also allow one to relax the convexity assumptions on the cost functions. 
Future work will focus on implementing 
more realistic simulation scenarios, to validate the applicability
of our results in biological networks.

An open problem is how to embed output constraints in the design, without compromising monotonicity properties. 
Moreover,  we only focused on
the analysis of an existing \gls{OFO} scheme; the design of new dynamics,   specifically tailored  to avoid timescale separation,  is a significant direction for future research.

\appendix

\subsection{Proof of Lemma~\ref{lem:SGT}}
\label{app:proofofSGT}
We explain how to obtain Lemma~\ref{lem:SGT} from \cite[Th.~2]{EncisoSontag_Global2006}. First,  we note that \cite{EncisoSontag_Global2006} identifies the dynamical system with its behavior, i.e., the function $x(t,x_0,u)$; while we consider the special case that the behavior is generated by the (nonsmooth) differential equation \eqref{eq:introsystem}; this results in a different notation, but no technical complication. We should only prove that the conditions in our Lemma~\ref{lem:SGT} ensure the assumptions H1, H2, H3, H4 in \cite[Th.~2]{EncisoSontag_Global2006}. H1 is satisfied because we restrict $\K_U$, $\K_X$, $\K_Y$ to be orthants. H2 and H3 are guaranteed by 1), as discussed in \cite[p.~553]{EncisoSontag_Global2006}. H4 is satisfied by conditions 1), 2) and 3). In particular, note that 1) and 3) imply that the continuous functions $k_x$ and $g$ must necessarily map bounded sets to bounded sets (as a consequence, $k_x$ is also completely continuous). In fact, $k_x$ and $g$  are defined over the closed sets $U$ and $X$: if $A\subseteq U$ is a bounded set, then $\operatorname{cl}(A) \subseteq U$ is compact, thus $k_x(A) \subseteq k_x(\operatorname{cl}(A))$ is bounded, since continuous functions map compact sets to compact sets; the same for $g$. 

\subsection{Proof of Proposition~\ref{prop:wellposedness}}
\label{app:proofpropWellposedness}
The closed-loop system can be written as the projected dynamical system $\dot \omega = \Pi_{\Omega} (\omega,\mathcal{A}(\omega)$), where $\omega = (x,u)$, $\Omega = \R^n \times \mathcal U$,
\begin{align}
\mathcal{A}(\omega) = \begin{bsmallmatrix}
        f(x,u)
        \\
       - \alpha  H (u) \nabla \Phi (u,g(x))
       \end{bsmallmatrix}.
\end{align}
By Assumptions~\ref{asm:plant}, \ref{asm:cost}, $\mathcal{A}$ is locally Lipschitz. Hence, for any initial condition in $\Omega$, there is a unique local solution $\omega(t) = (x(t),u(t))$, contained in $\Omega$ \cite{Cojocaru2003}. We next show  that $\omega(t)$ is contained in a compact set $\mathcal P$; then, $\mathcal A$ is Lipschitz on $\mathcal P$, hence a unique global solution exists (i.e., $\omega(t)$ can be extended for all $t\geq 0)$. In particular, $u(t) \in \mathcal U$, which is compact by Assumption~\ref{asm:cost}. Furthermore, $\mathcal U$ can be rewritten as $\{u \in \R^m \mid u^{\min} \preccurlyeq_{\K_U} u \preccurlyeq_{\K_U} u^{\max}$\}, for some vector $u^{\max},u^{\min} \in\R^m$. Therefore, in view of Assumption~\ref{asm:monotone}, we can apply \cite[Prop.~V.4]{Angeli_MCS_2003}  and conclude that $x^{\min} (t)\preccurlyeq_{\K_X} x(t) \preccurlyeq_{\K_X} x^{\max}(t) $, where $x^{\min} $ and $x^{\max}$ are the solutions of \eqref{eq:plant} obtained with constant input $u^{\min}$ and $u^{\max}$, respectively. However, $x^{\min} $ and $x^{\max}$ are bounded (and global) by Assumption~\ref{asm:plant}. Hence, $x(t)$ is  contained in a compact set $\mathcal P_x$. The conclusion follows by choosing $\mathcal P=\mathcal P_x \times \mathcal U$.

\subsection{Proof of Theorem~\ref{th:main}}
We start by showing that the  system \eqref{eq:clreformulatedx}-\eqref{eq:clreformulatedv} is monotone with respect to $K_U$, $\K_\Omega \coloneqq\K_X\times \K_V$, $K_Z$, and that it admits a well-defined steady-state map. 
These are known properties for the cascade of monotone systems, e.g.,  \cite[Prop.~IV.I]{Angeli_MCS_2003}; we include a proof for completeness. 

First, choose arbitrary $u(t) \seq_{\K_U} u'(t)$ for almost all $t$ and $\omega_0 = (x_0,v_0) \seq_{\K_\Omega} \omega_0' =  (x_0',v_0')$, i.e., $x_0 \seq_{\K_X} x_0'$ and $v_0 \seq_{\K_V} v_0'$. Then, by Assumption~\ref{asm:monotone}, $x(t)= x(t,x_0,u) \seq_{\K_X} x'(t) = x(t,x_0',u')$; hence, also $v(t,v_0,x) \seq_{\K_V} v(t,v_0',x')$ by Assumption~\ref{asm:core}(i). In other words, $\omega(t,\omega_0,u) \seq_{\K_\Omega} \omega(t,\omega_0',u') $. Also, for the output map $z = \phi(\omega) \coloneqq v$, it  holds that  $\phi (\omega) \seq_{\K_Z}  \phi(\omega')$. In conclusion,  \eqref{eq:clreformulatedx}-\eqref{eq:clreformulatedv} is monotone with respect to $K_U$, $\K_\Omega$, $K_Z$.

Second, for a constant input $u \equiv \bar u$, \eqref{eq:clreformulatedx} converges to $k_x(\bar u)$ by Assumption~\ref{asm:plant}, hence $v$ converges to $k_v(k_x(\bar u))$ by \cite[Th.~1]{EncisoSontag_Global2006} and Assumption~\ref{asm:core}. Since $k_v$ and $k_x$ are continuous, the steady-state map for the system \eqref{eq:clreformulatedx}-\eqref{eq:clreformulatedv} $k_\omega(u) = (k_x(u),k_v(k_x(u)))$ is also continuous. 

We have just proven the conditions 1), 2) and 3) in Lemma~\ref{lem:SGT} (in particular, note that $\K_Z = - \K_U$ by Assumption~\ref{asm:core}, and $\Omega$ is a box set).  Therefore we can infer global convergence of \eqref{eq:clreformulated} if also condition 4) holds, i.e., if the iteration \begin{align}\label{eq:iterationapp}u_{n+1}=k_z(u_n) \end{align} has a global attractor, where $k_z = k_v \circ k_x$ is the input-output steady-state map 
$k_z$ of \eqref{eq:clreformulated}. We next show that the iteration \eqref{eq:iterationapp} coincides with \eqref{eq:iterationmain}, thus concluding the proof. 

To start, the steady-state map $\bar v = k_v(x)$ of \eqref{eq:clreformulatedv} computed at any fixed $x$ has to satisfy   
\begin{align}
    \0 & = \Pi_{\mathcal U} \left(\bar v, - \alpha H(\bar v) \nabla \Phi (\bar v,g(x)) \right)
    \\ & =  \Pi_{\mathcal U} \left(\bar v, - \alpha \nabla \Phi_u(\bar v) - \alpha \nabla  k_y(\bar v)^\top  \nabla \Phi_y(g(x)) \right). \label{eq:usefulGradient}
\end{align}
Note that \eqref{eq:usefulGradient} corresponds to the equilibrium condition for the projected gradient flow applied to the function $\psi(\cdot) \coloneqq \Phi_u(\cdot)+ k_y(\cdot)^\top \nabla \Phi_y(g(x))$; hence, $\bar v = k_v(x)$ is the unique critical point of $\psi$ over $\mathcal U$ (otherwise the steady-state map would not be defined, violating  Assumption~\ref{asm:core}(i)). Since $\mathcal U$ is compact, $\psi$ achieves a minimum in $\mathcal U$, which must be a critical point. Hence, 
\begin{align}
    k_v(x) = \argmin_{v\in\mathcal U} \Phi_u(v)+ k_y(v)^\top \nabla \Phi_y(g(x)).
\end{align}
Therefore, by recalling that $g\circ k_x = k_y$ and since $z =v$, we can compute $k_z(u) = k_v(k_x(u))$ as 
\begin{align}
    k_z(u) = \argmin_{v\in \mathcal U} \Phi_u(v)+ k_y(v)^\top \nabla \Phi_y(k_y(u)).
\end{align}

\subsection{Preliminary results for the proof of Lemma \ref{lem:asmi}}

To prove Lemma~\ref{lem:asmi}, 
we will need some general results. In particular,  Lemma~\ref{lem:sensitivityprojected} proves continuity of the solutions of projected dynamical systems to perturbations. We use this result to prove
Theorem~\ref{th:monotoneProjected} below, which essentially shows that the projection on the tangent cone of a box set does not affect monotonicity with respect to an orthant.

\begin{lemma}\label{lem:sensitivityprojected} 
Consider a closed convex set $\mathcal{X} \subseteq \R^n$ and the projected dynamical system
    \begin{align}\label{eq:epsilonsystem}
        \dot x_\epsilon = \Pi_\mathcal{X} (x_\epsilon,q(x_\epsilon,t) - \epsilon v) 
    \end{align}
    where $\epsilon \geq 0$ (possibly $0$), $v\in\R^n$ is a constant vector, $q$ is Lipschitz in $x$ (uniformly in $t$) and continuous in $t$.
    Let  $x_\epsilon(t) = x_\epsilon(t, t_0,x_0)$ be the unique solution of \eqref{eq:epsilonsystem} with initial condition $(t_0,x_0)$.  Then, for any $T>0$, there is $L_T>0$ such that, for any $\epsilon, \epsilon' \geq 0$,
    \begin{align}
        \sup_{t\in [t_0,t_0+T]} \| x_\epsilon (t) - x_{\epsilon '} (t)\| \leq L_T |\epsilon - \epsilon'|. 
    \end{align}
\end{lemma}
\medskip 

\begin{proof}
    For almost all $t$, there is $n_\epsilon \in \mathrm N_\mathcal{X} (x_\epsilon)$ (where $\mathrm N_\mathcal{X}$ is the normal cone of $\mathcal X$) such that 
    \begin{align}
        \dot x_\epsilon  = q(x_\epsilon,t) -\epsilon v - n_\epsilon,
    \end{align}
    and similarly for $\dot x_{\epsilon'} = q(x_{\epsilon'},t) -\epsilon'v-n_\epsilon'$, for some $n_\epsilon'\in \mathrm N_\mathcal{X}(x_\epsilon') $. Define the error $e = x_\epsilon-x_\epsilon'$, and $w = n_\epsilon  - n_{\epsilon '}$. Hence, for almost all $t$,
    \begin{align*}    
        \frac{d}{dt} \frac{1}{2} \|e\|^2 & = \langle e, q(x_\epsilon,t) - q(x_{\epsilon'},t) \rangle  - \langle  e,w \rangle  -  \langle e, (\epsilon-\epsilon') v\rangle 
        \\ & \leq L_q \|e\|^2   + |\epsilon-\epsilon'| \| v\| \|e\|
\end{align*}
where we used Lipschitz continuity of $q$ and  that $\langle e,w \rangle \geq 0 $ by monotonicity of the normal cone. Hence, by Young's inequality, for almost all $t$,
\begin{align*}
    \frac{d}{dt} \|e\|^2 \leq (2L_q+1)\|e\|^2  + \|v\|^2 |\epsilon- \epsilon'|^2. 
\end{align*}
    Therefore, by  Gr\"onwall's inequality, and using that $x_\epsilon(t_0) = x_\epsilon'(t_0) = x_0$ and thus $e(t_0) = 0$, it holds for all $t$ that
    \begin{align*}
        \|e(t)\|^2& \leq \left( \frac { \|v\|^2 (e^{( 2L_q + 1) (t-t_0) } -1) } {2L_q+1} \right) |\epsilon - \epsilon '|^2 
        \\ &  \coloneqq L_t ^2  |\epsilon - \epsilon '| ^2. 
    \end{align*}
    which concludes the proof. 
\end{proof}

\medskip
\begin{theorem}\label{th:monotoneProjected}
    Let  $\mathcal X \subseteq \R^n$ be a box set, and let $\mathcal K_{U}$ and $\mathcal K_{X}$ be orthants. Consider the projected dynamical system $
        \dot x = \Pi_{\mathcal X}(x,q(x,u)) $,
    where $q$ is Lipschitz in $x$ on $\mathcal X$ (uniformly in $u$) and continuous in $u$, and the input $u:(t_0,\infty) \rightarrow \R^m$ is continuous. Denote by $x(t) = x(t,t_0,x_0,u)$ the  solution of this system with initial condition $x_0\in \mathcal X$. Assume that
    \begin{align}\label{eq:monotonicitycondition}
        u \succcurlyeq_{\mathcal K_U} u', x\succcurlyeq_{\mathcal K_X} x' \Rightarrow q(x,u) - q(x',u') \in \mathrm{T}_{\mathcal K_X}(x-x').
    \end{align}
    If $x_0 \succcurlyeq_{\mathcal{K}_X }x_0'$ and  and $u(t)\succcurlyeq_{\mathcal{K}_U }u'(t)$ for all $t\geq 0$, then  $x(t,x_0,u) \succcurlyeq_{\mathcal K_X} x(t,x_0',u')$, for all $t\geq 0$. 
\end{theorem}
\medskip

\begin{proof}
For simplicity of notation, and without loss of generality, we prove the result for the standard orderings, i.e., $\mathcal K_U =\R^m_{\geq 0}$, $\mathcal K_X=\R^n_{\geq 0}$.
Let $\mathcal X = [x_1^{\min},x_1^{\max}] \times \dots \times [x_n^{\min},x_n^{\max}]$. Note that, since $\mathcal X$ is a box, the projection on its tangent cone is order preserving, meaning that pointwise
    \begin{align}\label{eq:monotonicitypointwise}
       \begin{multlined}
        [x]_i \geq [x']_i, [v]_i \geq [v']_i  \\ \Rightarrow [\Pi_{\mathcal{X}}(x,v)]_i -   [\Pi_{\mathcal{X}}(x',v')]_i \in \mathrm{T_{\R_{\geq 0}}}([x]_i - [x']_i) 
        \end{multlined}
    \end{align}
    
    Let us consider the perturbed system $\dot x'_\epsilon = \Pi_{\mathcal X} (x'_\epsilon,q(x'_\epsilon,u') - \epsilon \boldsymbol{1})$, with $\epsilon>0$. We next show that $x'_\epsilon (t) = x'_\epsilon(t,t_0,x_0',u')$ satisfies $ x'_\epsilon (t) \preccurlyeq x(t) $ for all $t$ (independently of $\epsilon)$. Since this holds for any $\epsilon>0$, the conclusion follows by Lemma~\ref{lem:sensitivityprojected}, by taking the limit $\epsilon \rightarrow 0$.

    For contradiction, assume for some $t$ we have $x(t) \not\succcurlyeq x'_\epsilon(t)$, i.e.,   $[x(t)]_i < [x_\epsilon'(t)
]_i $ for some $i \in \{1,\dots,n\}$. Choose $   t^* = 
    \inf \{ t \mid x(t) \not\succcurlyeq x_\epsilon'(t)\}$, so that, by continuity of Carathéodory solutions, for some $\delta>0$,
        \begin{align}
       \begin{aligned}(\forall t \leq t^*) \ x(t) \succcurlyeq x_\epsilon'(t), \ [x(t^*)]_i = [x'_\epsilon(t^*)]_i,
       \\
       \ (\forall t \in (t^*,t^*+\delta)) \ [x(t)]_i < [x_\epsilon ' (t)]_i.
       \end{aligned}
    \end{align}
    We next show that there is $\rho>0$ such that $[x(t)]_i \geq [x_\epsilon'(t)]_i$ for all $t \in [t^*,t^*+\rho)$, contradicting the definition of $t^*$; this in turn implies $x(t) \succcurlyeq x_\epsilon'(t)$ for all $t \geq t^*$, hence the theorem. Note that $x(t) \in\mathcal X$ and $x_\epsilon'(t) \in \mathcal X$ for all $t$, and at $t^*$ it holds that $[q(x,u)]_i \geq [q(x_\epsilon',u')]_i $.  We consider the (five) possible different cases: 

    (A) At $t^*$, $[x]_i=[x_\epsilon']_i= x_i^{\max}$ and $[q(x,u)]_i >0$. Then, by continuity of $u$, $q$, and of the Carathéodory solution $x$, there is $\rho>0$ such that $[q(x,u)]_i >0$ for all $t \in [t^*,t^*+\rho)$. In this interval almost everywhere $ [\dot x ]_i 
    = [\Pi_\mathcal{X}(x,q(x,u))]_i \geq \min \{0,[q(x,u)]_i\} \geq 0$, where the first inequality is the definition of $\Pi_{\mathcal X}$ as in \eqref{eq:explicit_ptc}. Hence $[x(t)]_i = x_i^{\max} \geq [x_\epsilon'(t)]_i$ for all $t\in  [t^*,t^*+\rho)$. 

    (B) At $t^*$, $[x]_i=[x_\epsilon']_i= x_i^{\min}$ and $[q(x_\epsilon',u')]_i -\epsilon <0$: symmetric to case (A). 

    (C) At $t^*$, $[q(x,u)]_i\leq 0$ and  $[x]_i=[x_\epsilon']_i > x_i^{\min}$. By continuity of $u,q,x,x_\epsilon$, there is $\rho>0$ such that, for all $t\in [t^*,t^*+\rho)$, $[x]_i >x_i^{\min}$, $[x_\epsilon']_i > x_i^{\min}$, and $[\Pi_{\mathcal X} (x,q(x,u))]_i> [\Pi_{\mathcal X} (x_\epsilon',q(x_\epsilon',u')-\epsilon \boldsymbol 1)]_i$. Hence $[\dot x]_i > [\dot x_\epsilon ']$ almost everywhere in $[t^*,t^*+\rho)$, and thus $[x]_i \geq [x_\epsilon ']_i$. 

    (D) At $t^*$, $[q(x_\epsilon',u')]_i-\epsilon \geq 0$ and  $[x]_i=[x_\epsilon']_i < x_i^{\max}$: symmetric to case (C). 

    (E) At $t^*$, $[q(x_\epsilon',u')]_i - \epsilon  \geq  0$ and $[x]_i=[x_\epsilon']_i = x_i^{\max}$, or $[q(x,u)]_i \leq  0$ and $[x]_i=[x_\epsilon']_i = x_i^{\min}$: included in cases (A) and (B). 
\end{proof}

Note that the condition in \eqref{eq:monotonicitycondition} needs to be checked only if $[x]_i = [x']_i$, for some $i \in \{ 1, \dots, m\}$, since otherwise $[\mathrm T_{\mathcal K_X}(x-x')]_i = \R$. Theorem~\ref{th:monotoneProjected} only consider continuous inputs, while our definition of monotone systems is given for measurable inputs.  This is not an issue, since $x$ in \eqref{eq:clreformulatedv} is continuous, and Lemma~\ref{lem:SGT} still holds if restricting the attention to continuous inputs (see \cite{EncisoSontag_Global2006})

\subsection{Proof of Lemma~\ref{lem:asmi}}
    \emph{(i)} Since $m = 1$, we just need to check \eqref{eq:monotonicitycondition} for $v=v'$ (with $x,v$ in place of $u,x$). Let \begin{align}
        q(v,x) & \coloneqq - \alpha H(v) \nabla \Phi(v,g(x)) 
        \\ & = - \alpha \nabla \Phi_u(v)- \alpha \nabla k_y (v)^\top \nabla \Phi_y(g(x)).
    \end{align} Note that $k_y$ is nondecreasing, since it is the steady-state map of a monotone system \cite[Rem~V.2]{Angeli_MCS_2003}, hence $\nabla k_y(v)$ is nonnegative for all $v$; $\nabla \Phi_y$ is nondecreasing by convexity; $g(x)$ is also nondecreasing by Assumption~\ref{asm:monotone}. Therefore $x \seq x' \Rightarrow q(v,x) - q(v,x') 
    \leq 0$, and the conclusion follows by Theorem~\ref{th:monotoneProjected}.
    
    \emph{(ii)} The matrix $S$ must be nonnegative by \cite[Rem~V.2]{Angeli_MCS_2003} and Assumption~\ref{asm:monotone}; $g$ is also nondecreasing by Assumption~\ref{asm:monotone}.
    Under the conditions imposed, $v \preccurlyeq v' \Rightarrow \nabla \Phi_u(v)-\nabla \Phi_u(v) \in \mathrm T_{\R_{\geq 0}^m}(v-v')$ \cite{Angeli2004}; and $\nabla \Phi_y$ is nondecreasing (in the standard order, $x \succcurlyeq x' \Rightarrow \nabla \Phi_y(x) \succcurlyeq \nabla\Phi_y(x')$) \cite{Angeli2004}. 
    Therefore,  the condition \eqref{eq:monotonicitycondition} holds separately for the two components in $q(v,x)$, namely $-\alpha \nabla\Phi_u(v)$ and $-\alpha S^\top \nabla \Phi_y(g(x))$, and overall $x \succcurlyeq x', v\preccurlyeq v'  \Rightarrow q(v,x) -q(v',x') \in \mathrm{T}_{\R_{\leq 0}^m} (v-v')$. Hence the conclusion follows by Theorem~\ref{th:monotoneProjected}. 

    \subsection{Proof of Lemma~\ref{lem:asmii}} 
    \emph{(i)} If the function $\Psi$ is strictly convex, then its projected gradient flow \eqref{eq:clreformulatedv} converges to the unique minimizer (over the compact $\mathcal U$), for any fixed $x\in \R^n$. Furthermore, $x\mapsto \argmin_{u\in\mathcal U} \Psi(u,x)$ is a continuous function whenever $\mathcal U$ is compact and $\Psi$ is continuous and has a unique minimizer for each $x$, see \cite[Lem.~3]{Bianchi_Mon_CDC_2022} for an analogous proof (where strong convexity  replaces the role of compactness). 
    
    \emph{(ii)} This is a sufficient condition for (i). 

   \subsection{Proof of Lemma~\ref{lem:asmiii}}
For any $w,w' \in \mathcal U$, let $u = \argmin_{v\in \mc{U}} \tilde\Psi (v,w)$ and $u' = \argmin_{v\in \mc{U}} \tilde\Psi (v,w')$. We have the first order optimality relations 
\begin{align*}
    \langle \nabla_u \tilde\Psi (u,w), u'-u \rangle & \geq 0
    \\ 
     \langle \nabla_u \tilde\Psi (u',w'), u-u' \rangle & \geq 0.
\end{align*}
Summing up, and adding and subtracting $ \nabla_u \tilde\Psi (u,w')$ on the left side, we obtain
$\langle \nabla_u \tilde\Psi (u,w) -\nabla_u \tilde\Psi (u,w'), u'-u\rangle  + \langle \nabla_u \tilde\Psi (u,w')  - \nabla_u \tilde\Psi (u',w'), u'-u \rangle  \geq 0$. By using the Cauchy--Schwarz inequality and smoothness for the first addend, and strong convexity for the second, we obtain
\begin{align*}
      & \ell \| w -w'\|\|u-u'\| - \mu \| u - u'\| ^2 \geq 0
      \\ \Rightarrow  &  \| u-u'\| \leq \textstyle  \frac{\ell}{\mu} \|w -w'\|. \numberthis \label{eq:contractivity}
\end{align*}

For the iteration in \eqref{eq:iterationmain}, note that it can be rewritten as
\begin{align*}
    u_{n+1} = \argmin_{u\in\mathcal U} \tilde \Psi (u,u_n).
\end{align*}
Let $u^\star$ be a fixed point of this iteration (one exists by existence of solutions to \eqref{eq:opt1}). Then by \eqref{eq:contractivity} we have
\begin{align} \| u_{n+1} - u^\star\| \leq \textstyle \frac{\ell}{\mu} \| u_{n} - u^\star\|, \end{align}
and recursively 
\begin{align}
    \|u_{n+1} - u^\star\| \leq \left(\frac{\ell}{\mu}\right)^{n+1} \|u_0-u^\star\|,
\end{align}
which shows that $u_n$ converges to $u^\star$, since $\ell/\mu <1$ (and thus the fixed point $u^\star$ is unique).

\subsection{Proof of Corollary~\ref{cor:linear}}
The conditions in the statement imply Assumption~\ref{asm:core}, hence the conclusion follows by Theorem~\ref{th:main}. Specifically, satisfaction of Assumption~\ref{asm:core}(i) follows by Lemma~\ref{lem:asmi}(ii). Satisfaction of Assumption~\ref{asm:core}(ii) follows by Lemma~\ref{lem:asmii}(ii). Finally, note that $\tilde \Psi(u,\bar u) = \Phi_u(u)+k_y(u)^\top \nabla \Phi_y(k_y(\bar u ) ) = \Phi_u(u)+  2\beta_y [Su]^\top [S\bar u]  $ is $2\beta_u$-strongly convex in $u$  and furthermore $\nabla_u \tilde \Psi(u,\bar u) = 2\beta_y S^\top S \bar u $ is Lipschitz in $\bar u$  with parameter ($ 2\beta_y   \|S\|^2$) for any fixed $ u$. Then satisfaction of Assumption~\ref{asm:core}(iii) can be inferred by Lemma~\ref{lem:asmiii}.  


\bibliographystyle{IEEEtran}
\bibliography{library}

\begin{thebibliography}{10}
\providecommand{\url}[1]{#1}
\csname url@samestyle\endcsname
\providecommand{\newblock}{\relax}
\providecommand{\bibinfo}[2]{#2}
\providecommand{\BIBentrySTDinterwordspacing}{\spaceskip=0pt\relax}
\providecommand{\BIBentryALTinterwordstretchfactor}{4}
\providecommand{\BIBentryALTinterwordspacing}{\spaceskip=\fontdimen2\font plus
\BIBentryALTinterwordstretchfactor\fontdimen3\font minus \fontdimen4\font\relax}
\providecommand{\BIBforeignlanguage}[2]{{%
\expandafter\ifx\csname l@#1\endcsname\relax
\typeout{** WARNING: IEEEtran.bst: No hyphenation pattern has been}%
\typeout{** loaded for the language `#1'. Using the pattern for}%
\typeout{** the default language instead.}%
\else
\language=\csname l@#1\endcsname
\fi
#2}}
\providecommand{\BIBdecl}{\relax}
\BIBdecl

\bibitem{Hauswirth2021}
A.~Hauswirth, Z.~He, S.~Bolognani, G.~Hug, and F.~D{\"o}rfler, ``Optimization algorithms as robust feedback controllers,'' \emph{Annual Reviews in Control}, vol.~57, p. 100941, 2024.

\bibitem{simpson2020stability}
J.~W. Simpson-Porco, ``On stability of distributed-averaging proportional-integral frequency control in power systems,'' \emph{IEEE Control Systems Letters}, vol.~5, no.~2, pp. 677--682, 2020.

\bibitem{chen2020distributed}
X.~Chen, C.~Zhao, and N.~Li, ``Distributed automatic load frequency control with optimality in power systems,'' \emph{IEEE Transactions on Control of Network Systems}, vol.~8, no.~1, pp. 307--318, 2020.

\bibitem{Emiliano2016}
E.~Dall'Anese, ``Optimal power flow pursuit,'' in \emph{Proceedings of American Control Conference}, 2016, pp. 1767--1767.

\bibitem{Ortmann_Deployment_2023}
L.~Ortmann, C.~Rubin, A.~Scozzafava, J.~Lehmann, S.~Bolognani, and F.~Dörfler, ``Deployment of an online feedback optimization controller for reactive power flow optimization in a distribution grid,'' in \emph{2023 IEEE PES Innovative Smart Grid Technologies Europe (ISGT EUROPE)}, 2023, pp. 1--6.

\bibitem{Belgioiosogiuseppe2021}
G.~Belgioioso, D.~Liao-McPherson, M.~H. de~Badyn, S.~Bolognani, J.~Lygeros, and F.~D\"{o}rfler, ``Sampled-data online feedback equilibrium seeking: Stability and tracking,'' in \emph{Proceedings of 60th IEEE Conference on Decision and Control}, 2021, pp. 2702--2708.

\bibitem{wang2011control}
J.~Wang and N.~Elia, ``A control perspective for centralized and distributed convex optimization,'' in \emph{2011 50th IEEE Conference on Decision and Control and European Control Conference}, 2011, pp. 3800--3805.

\bibitem{bianchin2021time}
G.~Bianchin, J.~Cort{\'e}s, J.~I. Poveda, and E.~Dall’Anese, ``Time-varying optimization of {LTI} systems via projected primal-dual gradient flows,'' \emph{IEEE Transactions on Control of Network Systems}, vol.~9, no.~1, pp. 474--486, 2021.

\bibitem{Hauswirthadrian2021}
A.~Hauswirth, S.~Bolognani, G.~Hug, and F.~D\"{o}rfler, ``Timescale separation in autonomous optimization,'' \emph{IEEE Transactions on Automatic Control}, vol.~66, no.~2, pp. 611--624, 2021.

\bibitem{Belgioioso_FES_2024}
G.~Belgioioso, D.~Liao-McPherson, M.~H. de~Badyn, S.~Bolognani, R.~S. Smith, J.~Lygeros, and F.~Dörfler, ``Online feedback equilibrium seeking,'' \emph{IEEE Transactions on Automatic Control}, pp. 1--16, 2024.

\bibitem{menta2018stability}
S.~Menta, A.~Hauswirth, S.~Bolognani, G.~Hug, and F.~D{\"o}rfler, ``Stability of dynamic feedback optimization with applications to power systems,'' in \emph{Proceedings of 56th Annual Allerton Conference on Communication, Control, and Computing}.\hskip 1em plus 0.5em minus 0.4em\relax IEEE, 2018, pp. 136--143.

\bibitem{Colombino2020}
M.~Colombino, E.~Dall'Anese, and A.~Bernstein, ``Online optimization as a feedback controller: Stability and tracking,'' \emph{IEEE Transactions on Control of Network Systems}, vol.~7, pp. 422--432, 3 2020.

\bibitem{Picallo2023}
M.~Picallo, S.~Bolognani, and F.~Dörfler, ``Sensitivity conditioning: Beyond singular perturbation for control design on multiple time scales,'' \emph{IEEE Transactions on Automatic Control}, vol.~68, no.~4, pp. 2309--2324, 2023.

\bibitem{Bianchi_InfOfo}
M.~Bianchi and F.~Dörfler, ``A stability condition for online feedback optimization without timescale separation,'' p. 101308, 2025.

\bibitem{Angeli2014}
D.~Angeli, G.~A. Enciso, and E.~D. Sontag, ``A small-gain result for orthant-monotone systems under mixed feedback,'' \emph{Systems \& Control Letters}, vol.~68, pp. 9--19, 2014.

\bibitem{Angeli2004Detection}
D.~Angeli, J.~E. Ferrell, and E.~D. Sontag, ``Detection of multistability, bifurcations, and hysteresis in a large class of biological positive-feedback systems,'' \emph{Proceedings of the National Academy of Sciences of the United States of America}, vol. 101, 2004.

\bibitem{Nikolaev2016}
E.~V. Nikolaev and E.~D. Sontag, ``Quorum-sensing synchronization of synthetic toggle switches: A design based on monotone dynamical systems theory,'' \emph{PLoS Computational Biology}, vol.~12, 2016.

\bibitem{haddad2010nonnegative}
W.~M. Haddad, V.~Chellaboina, and S.~G. Nersesov, \emph{Nonnegative and Compartmental Dynamical Systems}.\hskip 1em plus 0.5em minus 0.4em\relax Princeton, NJ: Princeton University Press, 2010.

\bibitem{Angeli_MCS_2003}
D.~Angeli and E.~Sontag, ``Monotone control systems,'' \emph{IEEE Transactions on Automatic Control}, vol.~48, no.~10, pp. 1684--1698, 2003.

\bibitem{Cojocaru2003}
M.-G. Cojocaru and L.~B. Jonker, ``Existence of solutions to projected differential equations in {Hilbert} spaces,'' \emph{Proceedings of the American Mathematical Society}, vol. 132, 2003.

\bibitem{EncisoSontag_Global2006}
G.~A. Enciso and E.~D. Sontag, ``Global attractivity, {I/O} monotone small-gain theorems, and biological delay systems,'' \emph{Discrete and Continuous Dynamical Systems}, vol.~14, no.~3, pp. 549--578, 2006.

\bibitem{Angeli2004}
D.~Angeli and E.~Sontag, ``Interconnections of monotone systems with steady-state characteristics,'' in \emph{Optimal Control, Stabilization and Nonsmooth Analysis}.\hskip 1em plus 0.5em minus 0.4em\relax Springer Berlin Heidelberg, 2004, pp. 135--154.

\bibitem{Malisoff2006}
M.~Malisoff and P.~D. Leenheer, ``A small-gain theorem for monotone systems with multivalued input-state characteristics,'' \emph{IEEE Transactions on Automatic Control}, vol.~51, 2006.

\bibitem{Uhlendorf2012}
J.~Uhlendorf, A.~Miermont, T.~Delaveau, G.~Charvin, F.~Fages, S.~Bottani, G.~Batt, and P.~Hersen, ``Long-term model predictive control of gene expression at the population and single-cell levels,'' \emph{Proceedings of the National Academy of Sciences}, vol. 109, no.~35, pp. 14\,271--14\,276, 2012.

\bibitem{Bianchi_Mon_CDC_2022}
M.~Bianchi and S.~Grammatico, ``Nash equilibrium seeking under partial decision information: {Monotonicity,} smoothness and proximal-point algorithms,'' in \emph{2022 61st IEEE Conference on Decision and Control (CDC)}, 2022, pp. 5080--5085.

\end{thebibliography}

\end{document}